\documentclass[a4]{amsart}

\input xypic 
\input xy 
\xyoption{all} 
\usepackage{amssymb} 
\usepackage{bbm}
\usepackage{tikz-cd}
\usepackage[bookmarksnumbered, bookmarksopen,
colorlinks,citecolor=blue,linkcolor=blue,backref]{hyperref}

\newtheorem{theorem}{Theorem}[section]
\newtheorem{proposition}[theorem]{Proposition}
\newtheorem{lemma}[theorem]{Lemma}
\newtheorem{corollary}[theorem]{Corollary}

\theoremstyle{definition}
\newtheorem{definition}[theorem]{Definition}

\newtheorem{example}[theorem]{Example}

\newtheorem{conjecture}[theorem]{Conjecture}

\newcommand{\conn}{\ensuremath{\#}} 
 
\newcommand{\calL}{\ensuremath{\mathcal{L}}} 
\newcommand{\calP}{\ensuremath{\mathcal{P}}} 

\newcommand{\hlgy}[1]{\ensuremath{H_{*}(#1)}}

\newcommand{\cohlgy}[1]{\ensuremath{H^{*}(#1)}}

\newcounter{bean}
\newenvironment{letterlist}{\begin{list}{\rm ({\alph{bean}})}
      {\usecounter{bean}\setlength{\rightmargin}{\leftmargin}}}
      {\end{list}}

\newcommand{\namedright}[3]{\ensuremath{#1\stackrel{#2}
 {\longrightarrow}#3}}
\newcommand{\nameddright}[5]{\ensuremath{#1\stackrel{#2}
 {\longrightarrow}#3\stackrel{#4}{\longrightarrow}#5}}

\newcommand{\larrow}{\relbar\!\!\relbar\!\!\rightarrow}
\newcommand{\llarrow}{\relbar\!\!\relbar\!\!\larrow}

\newcommand{\llnamedright}[3]{\ensuremath{#1\stackrel{#2}
 {\llarrow}#3}}

\newcommand{\qqed}{\hfill\Box}

\begin{document}


\title[Exponential growth] 
   {Exponential growth in the rational homology of free loop spaces and in torsion homotopy groups} 

\author{Ruizhi Huang} 
\address{Institute of Mathematics, Academy of Mathematics and Systems Science, 
   Chinese Academy of Sciences, Beijing 100190, China} 
\email{haungrz@amss.ac.cn} 
 \urladdr{https://sites.google.com/site/hrzsea/}

\author{Stephen Theriault}
\address{School of Mathematics, University of Southampton, Southampton 
   SO17 1BJ, United Kingdom}
\email{S.D.Theriault@soton.ac.uk}



\begin{abstract} 
Using integral methods we recover and generalize some results by F\'{e}lix, Halperin 
and Thomas on the growth of the rational homology groups of free loop spaces, and 
obtain a new family of spaces whose $p$-torsion in homotopy groups grows exponentially 
and satisfies Moore's Conjecture for all but finitely many primes. In view of the results, we conjecture that there should be a strong connection between exponential growth in the rational homotopy groups and the $p$-torsion homotopy groups for any prime $p$. 
\end{abstract}

\maketitle

\section{Introduction} 
Let $X$ be a simply-connected space. Its \emph{free loop space} $\calL X$ is the 
space ${\rm Map}(S^1,X)$ of continuous maps from $S^{1}$ to $X$ and its 
\emph{based loop space} $\Omega X$ is the space ${\rm Map}_{\ast}(S^{1},X)$ of 
pointed continuous maps from $S^{1}$ to $X$. They are related via a fibration 
\(\nameddright{\Omega X}{}{\calL X}{e}{X}\) 
where $e$ evaluates a loop at the basepoint. This paper is concerned with growth 
in the rational homology of $\calL X$ and growth in the homotopy groups of $\Omega X$. 

Gromov \cite{G} conjectured that when $X$ is a closed manifold then 
$H_\ast(\mathcal{L}X;\mathbb{Q})$ almost always grows exponentially. This has an 
important consequence in Riemannian geometry due to a theorem of Gromov, improved 
by Ballman and Ziller~\cite{G,BZ}, which shows that the rate of growth in the dimension 
of $H_{\ell}(\calL M;\mathbb{Q})$ can be used to give a lower bound on the number of 
geometrically distinct closed geodesics on a simply-connected closed Riemannian 
manifold $M$. 

Vigu\'{e}-Poirrier made a more general conjecture. A finite type space $X$ of finite 
Lusternik-Schnirrelman category is \emph{rationally  elliptic} if the dimension of 
$\pi_{\ast}(X)\otimes\mathbb{Q}$ is finite and is \emph{rationally hyperbolic} otherwise. 
Vigu\'{e}-Poirrier \cite{V} conjectured that if $X$ is rationally hyperbolic then 
$H_\ast(\mathcal{L}X;\mathbb{Q})$ grows exponentially. This conjecture has 
been proved for a finite wedge of spheres \cite{V}, for a non-trivial connected sum 
of closed manifolds \cite{L2} and in the case when $X$ is coformal \cite{L1}. 

To be precise, a graded vector space $V=\{V_i\}_{i\geq 0}$ of finite type {\it grows exponentially} 
if there exist constants $1<C_1<C_2<\infty$ such that for some $K$
\[
C_1^k \leq r_k\leq C_2^k, ~~\ \ \ \ \ ~\quad k\geq K,
\] 
where $r_k=\sum\limits_{i\leq k} {\rm dim} V_i$. The {\it log index} of $V$ is defined by
\[
{\rm log}~{\rm index}(V)=\mathop{{\rm lim}~{\rm sup}}\limits_{i} \frac{{\rm log}~({\rm dim}V_i)}{i}.
\]
Notice that when $V$ grows exponentially $0<{\rm log}~{\rm index}(V)<\infty$.
For a topological space $X$, let
${\rm log}~{\rm index}(\pi_\ast(X))={\rm log}~{\rm index}(\pi_{\geq 2}(X)\otimes \mathbb{Q})$. 
In particular, if $X$ is rationally elliptic then ${\rm log}~{\rm index}(\pi_\ast(X))=-\infty$ 
and if $X$ is rationally hperbolic then ${\rm log}~{\rm index}(\pi_\ast(X))>0$.

F\'{e}lix, Halperin and Thomas \cite{FHT4} introduced a much stronger version of exponential growth. 

\begin{definition}\label{conexpdef}
A graded vector space $V=\{V_i\}_{i\geq 0}$ of finite type has {\it controlled exponential growth} if $0<{\rm log}~{\rm index}(V)<\infty$ and for each $\lambda>1$ there is an infinite sequence $n_1<n_2<\cdots$ such that $n_{i+1}<\lambda n_i$ for $i\geq 0$, and ${\rm dim}V_{n_i}=e^{\alpha_in_i}$ with $\alpha_i\rightarrow {\rm log}~{\rm index}(V)$.
\end{definition}
They then observed that for any simply-connected space $X$ with rational homology of finite type
\begin{equation} 
  \label{logcompare} 
{\rm log}~{\rm index}(H_\ast(\mathcal{L}X;\mathbb{Q}))\leq {\rm log}~{\rm index}(\pi_\ast(X))={\rm log}~{\rm index}(H_\ast(\Omega X;\mathbb{Q})) 
\end{equation} 
which led to a further refinement of exponential growth. 
 
\begin{definition}\label{goodexpdef}
Let $X$ be a simply-connected space with rational homology of finite type such that ${\rm log}~{\rm index}(H_\ast(\Omega X;\mathbb{Q}))\in (0,\infty)$. Then $\mathcal{L}X$ has {\it good exponential growth} if $H_\ast(\mathcal{L}X;\mathbb{Q})$ has controlled exponential growth and 
\[
{\rm log}~{\rm index}(H_\ast(\mathcal{L}X;\mathbb{Q}))={\rm log}~{\rm index}(H_\ast(\Omega X;\mathbb{Q})).
\]
\end{definition}

F\'{e}lix, Halperin and Thomas~\cite{FHT4,FHT5} went on to give conditions that guaranteed 
good exponential growth and provided several families of examples that have this property. This 
was done using a detailed analysis of the relevant Sullivan models. 

In this paper we reformulate and generalize some of F\'{e}lix, Halperin and Thomas' results 
using \emph{integral} methods rather than rational ones, based on recent work in~\cite{BT1,BT2,T}. 
The integral approach also lets us produce results on the growth of torsion in $\pi_{\ast}(X)$ 
for all but finitely many primes, which relates to Moore's Conjecture. 

Let $p$ be a prime. The \emph{$p$-primary homotopy exponent} of a space $X$ is the least power 
of $p$ that annihilates the $p$-torsion in $\pi_{\ast}(X)$. If this power is $r$ write $\exp_{p}(X)=p^{r}$, 
and if no such power exists we say $X$ has no homotopy exponent at $p$. Moore's Conjecture posits 
a deep relationship between homotopy exponents and the number of rational homotopy groups. 

\begin{conjecture}[Moore] 
\label{Moore}
Let $X$ be a finite simply-connected $CW$-complex. Then the following are equivalent: 
\begin{letterlist} 
   \item $X$ is rationally elliptic; 
   \item $\exp_{p}(X)<\infty$ for some prime $p$; 
   \item $\exp_{p}(X)<\infty$ for all primes $p$. 
\end{letterlist} 
\end{conjecture} 

Conjecture~\ref{Moore} is known to hold in a variety of special cases, including finite $H$-spaces~\cite{Lo,St}, 
$H$-spaces with finitely generated homology~\cite{CPSS}, odd primary Moore spaces~\cite{N}, 
torsion-free suspensions~\cite{Se}, and generalized moment-angle complexes~\cite{HST}.  

Moore's Conjecture asserts that rationally hyperbolic spaces have torsion homotopy groups 
of arbitrarily high order. But it says nothing about the rate of growth of the $p$-torsion 
in the homotopy groups. To address this the first author and Wu~\cite{HW} introduced 
the notion of $\mathbb{Z}/p^r$-hyperbolicity in analogy with rational hyperbolicity. 

\begin{definition} 
A $p$-local space $X$ is \emph{$\mathbb{Z}/p^r$-hyperbolic} if the number of 
$\mathbb{Z}/p^{r}$-summands in $\pi_{\ast}(X)$ has exponential growth, that is, 
\[\liminf_{n}\frac{{\rm log} ~t_{n}}{n}>0\] 
where $t_{n}$ is the number of $\mathbb{Z}/p^{r}$-summands in $\mathop{\bigoplus}\limits_{m\leq n}\pi_{m}(X)$. 
\end{definition} 

The following theorem summarizes our main results, stated in a less technical but slightly 
weakened form. The full statements can be found in Theorem~\ref{case1gexpthm}, 
Theorem~\ref{rhoinfthm} and Corollary~\ref{mostpMoore}. For a $CW$-complex $Z$, 
the \emph{rational homotopy Lie algebra} of $Z$ is $L_Z:=\pi_\ast(\Omega Z)\otimes \mathbb{Q}$, where 
the Lie algebra structure is induced by the Samelson product. 

\begin{theorem}\label{main}
Let $\Sigma A\stackrel{f}{\rightarrow} Y\stackrel{h}{\rightarrow} Z$ be a homotopy cofibration of simply-connected finite $CW$-complexes such that $A$ and $Z$ are not rationally contractible and $\Omega h$ has a right homotopy inverse. The following hold: 
\begin{letterlist} 
   \item if ${\rm log}~{\rm index}(\pi_\ast(Z))< {\rm log}~{\rm index}(\pi_\ast(Y))$ 
   and ${\rm log}~{\rm index}(H_{\ast}(\Omega Y;\mathbb{Q})\in (0,\infty)$, then $\mathcal{L}Y$ has good exponential growth;  
   \item if $Z$ is rationally hyperbolic with a finitely generated rational homotopy Lie algebra~$L_Z$, then $\calL Y$ has good exponential growth; 
   \item if $H_{\ast}(Y;\mathbb{Z})$ is $p$-torsion free for a prime $p$ that is sufficiently large, then
    $Y$ is rationally hyperbolic and $\mathbb{Z}/p^r$-hyperbolic for all $r\geq 1$. 
\end{letterlist} 
\end{theorem}

It is worth noting that F\'{e}lix, Halperin and Thomas studied good exponential growth via a fibration in \cite[Theorem 2 and Theorem 3]{FHT4} while Theorem \ref{main} starts instead from a cofibration. 

Theorem \ref{main} suggests that there may be a strong connection between rational hyperbolicity and $\mathbb{Z}/p^{r}$-hyperbolicity. We therefore propose an amplification of part of Moore's conjecture.
\begin{conjecture}
Let $X$ be a finite simply-connected $CW$-complex. If $X$ is rationally hyperbolic then it is 
$\mathbb{Z}/p^r$-hyperbolic for all primes $p$ and positive integers $r$. 
\end{conjecture}

The paper is organized as follows. In Section \ref{sec:caseII}, we prove Theorem \ref{case1gexpthm} as a generalization of ~\cite[Theorem 4]{FHT4} characterizing good exponential growth via a cofibration instead of a fibration. In Section~\ref{sec:caseII}, we show that the manifolds considered in \cite[Section 2]{BT1} have good exponential growth. Section~\ref{sec:caseIII} is devoted to proving Theorem~\ref{rhoinfthm}, using the analytic condition considered in~\cite{FHT5} to generalize~\cite[Theorem 1.3]{FHT5}. We also provide an example, the so-called general connected sum, that partially generalizes~\cite[Theorem 3]{L2} and \cite[Theorem 1.4]{FHT5}. In Section \ref{sec:local}, we turn to exponential growth in torsion homotopy groups and prove Corollary \ref{mostpMoore}.

\bigskip

\noindent{\bf Acknowledgements.}
The first author was supported in part by the National Natural Science Foundation of China (Grant nos. 11801544 and 12288201), the National Key R\&D Program of China (No. 2021YFA1002300), the Youth Innovation Promotion Association of Chinese Academy Sciences, and the ``Chen Jingrun'' Future Star Program of AMSS.

The authors sincerely thank the referee for closely reading the manuscript and making several valuable comments and suggestions, which have greatly improved the exposition of the paper.

\section{Exponential growth in $\hlgy{\calL Y;\mathbb{Q}}$: Case I}
\label{sec:caseI}  

We will need two preliminary results from~\cite{FHT4} and one from~\cite{BT2}. 

\begin{theorem}[Theorem 1 in \cite{FHT4}]\label{FHTthm1}
Let $X$ be a simply-connected wedge of spheres of finite type such that ${\rm log}~{\rm index}(H_\ast(\Omega X;\mathbb{Q}))\in (0,\infty)$. Then $\mathcal{L}X$ has good exponential growth.~$\qqed$  
\end{theorem}

As context, note that a simply-connected space that is rationally of finite type need not have the property that 
${\rm log}~{\rm index}(H_{\ast}(\Omega X;\mathbb{Q}))<\infty$. For example, as $e^{e^{x}}$ 
grows faster than exponentially, if~$X$ is the wedge of spheres $X=\mathop{\bigvee}\limits_{n=2}^{\infty} X_{n}$ 
where $X_{n}$ is a wedge of $e^{e^{n}}$ copies of $S^{n}$, then $X$ is rationally of finite 
type but $H_{\ast}(X;\mathbb{Q})$ has faster than exponential growth and so does 
$H_{\ast}(\Omega X;\mathbb{Q})$.

\begin{theorem}[Theorem 3 in \cite{FHT4}]\label{FHTthm3}
Let $F\rightarrow Y\rightarrow  Z$ be a fibration between simply-connected spaces with rational homology of finite type. If ${\rm log}~{\rm index}(\pi_\ast(Z))<{\rm log}~{\rm index}(\pi_\ast(Y))$, then $\mathcal{L}Y$ has good exponential growth if and only if $\mathcal{L}F$ does. In this case $H_\ast(\mathcal{L}Y;\mathbb{Q})$ and $H_\ast(\mathcal{L}F;\mathbb{Q})$ have the same log index.~$\qqed$ 
\end{theorem} 

\begin{theorem}[Proposition 3.5 in~\cite{BT2}] 
\label{BTtheorem} 
Suppose that 
\(\nameddright{\Sigma A}{f}{Y}{h}{Z}\) 
is a homotopy cofibration of simply-connected spaces and $\Omega h$ has a right homotopy inverse. 
Then there is a homotopy fibration 
\[\nameddright{(\Omega Z\wedge\Sigma A)\vee\Sigma A}{}{Y}{h}{Z}\] 
which splits after looping to give a homotopy equivalence 
\[
 \hspace{5cm}
 \Omega Y\simeq\Omega Z\times\Omega((\Omega Z\wedge\Sigma A)\vee\Sigma A).
 \hspace{5cm}\Box\] 
\end{theorem}

It will be useful to record a growth result related to Theorem~\ref{BTtheorem}.  

\begin{lemma} 
   \label{BTgrowth} 
   If $\Sigma A$ and $Z$ are simply-connected spaces that are not rationally contractible then 
   ${\rm log}~{\rm index}(H_{\ast}(\Omega((\Omega Z\wedge\Sigma A)\vee\Sigma A)))>0$. 
\end{lemma} 

\begin{proof} 
As any suspension is rationally 
a wedge of spheres, $(\Omega Z\wedge\Sigma A)\vee\Sigma A$ is rationally a wedge of spheres. 
Since $\Sigma A$ is not rationally contractible, it has at least one sphere as a wedge summand. 
Since~$Z$ is not rationally contractible and simply-connected, $\Omega Z$ is not rationally contractible, 
and therefore $\Omega Z\wedge\Sigma A$ also has at least one sphere 
as a wedge summand. Thus $(\Omega Z\wedge\Sigma A)\vee\Sigma A$ is rationally a 
wedge of at least two spheres, implying that it is rationally hyperbolic, and therefore 
${\rm log}~{\rm index}(\pi_{\ast}((\Omega Z\wedge\Sigma A)\vee\Sigma A))>0$. Then, 
by~(\ref{logcompare}), 
${\rm log}~{\rm index}(H_\ast(\Omega(\Omega Z\wedge\Sigma A)\vee\Sigma A);\mathbb{Q}))>0$. 
\end{proof}

We introduce two variations of an inert map. Given a homotopy cofibration 
$\Sigma A\stackrel{f}{\rightarrow} Y\stackrel{h}{\rightarrow} Z$ 
of simply-connected spaces, the map $f$ is {\it inert} if $h$ is surjective in rational homotopy. 
This generalizes the classical notion of an inert map defined by F\'{e}lix, Halperin and Thomas \cite{FHT1}, who 
considered the case when $A=S^{k}$ for some $k\geq 1$ and $f$ is a cell attachment.  
As spaces are simply-connected, the surjectivity of $h$ in rational homotopy impies that $\Omega h$ is 
also surjective in rational homotopy, so as loop spaces split rationally as products of Eilenberg-Mac Lane spaces \cite[Chapter 16 (c)]{FHT2}, $\Omega h$ has a right homotopy inverse. Thus $f$ is inert in our sense 
if and only if $\Omega h$ has a right homotopy inverse rationally. 
Moreover, if $f$ is inert it follows that ${\rm log}~{\rm index}(\pi_\ast(Z))\leq {\rm log}~{\rm index}(\pi_\ast(Y))$. We call $f$ {\it strongly inert} if it is inert and ${\rm log}~{\rm index}(\pi_\ast(Z))< {\rm log}~{\rm index}(\pi_\ast(Y))$.
The following theorem is a stronger form of Theorem \ref{main}~(a). 

\begin{theorem}\label{case1gexpthm}
Let $\Sigma A\stackrel{f}{\rightarrow} Y\stackrel{h}{\rightarrow} Z$ be a homotopy cofibration of 
simply-connected spaces that are rationally of finite type, and suppose that 
${\rm log}~{\rm index}(H_\ast(\Omega Y;\mathbb{Q}))\in (0,\infty)$. If $f$ is strongly inert 
then~$\mathcal{L}Y$ has~good exponential growth. 
\end{theorem}

\begin{proof} 
By definition, $f$ being strongly inert means that 
${\rm log}~{\rm index}(\pi_\ast(Y))>{\rm log}~{\rm index}(\pi_\ast(Z))$. 
In particular, this implies that $\Sigma A$ is not rationally contractible. 
There are two cases depending on whether $Z$ is rationally contractible.  

If $Z$ is rationally contractible, then $Y\simeq \Sigma A$ is rationally a simply-connected wedge 
of spheres of finite type. By hypothesis, 
${\rm log}~{\rm index}(H_\ast(\Omega Y;\mathbb{Q}))\in (0,\infty)$, 
so Theorem \ref{FHTthm1} implies that $\mathcal{L}Y$ has good exponential growth. 

If $Z$ is not rationally contractible, then the hypothesis that $f$ is inert implies that $\Omega h$ 
has a right homotopy inverse rationally. Theorem~\ref{BTtheorem} then implies that there is 
a rational homotopy fibration 
 \begin{equation} 
   \label{BTcase1} 
   \nameddright{(\Omega Z\wedge\Sigma A)\vee\Sigma A}{}{Y}{h}{Z} 
\end{equation} 
that splits after looping to give a homotopy equivalence 
\[\Omega Y\simeq\Omega Z\times\Omega((\Omega Z\wedge\Sigma A)\vee\Sigma A).\] 
For convenience, let $W=(\Omega Z\wedge\Sigma A)\vee\Sigma A$.
With respect to~(\ref{BTcase1}), $f$ being strongly inert means that 
${\rm log}~{\rm index}(\pi_\ast(Y))>{\rm log}~{\rm index}(\pi_\ast(Z))$. 
Theorem~\ref{FHTthm3} therefore implies that $\mathcal{L}Y$ has good exponential growth if and only if 
$\mathcal{L} W$ does. Since $\Sigma A$ and $Z$ 
are both not rationally contractible, Lemma~\ref{BTgrowth} implies that 
${\rm log}~{\rm index}(H_\ast(\Omega W;\mathbb{Q}))>0$. On the other hand, since $\Omega W$ 
is a retract of $\Omega Y$ and ${\rm log}~{\rm index}(H_\ast(\Omega Y;\mathbb{Q}))<\infty$
by hypothesis, we obtain  ${\rm log}~{\rm index}(H_\ast(\Omega W;\mathbb{Q}))<\infty$. 
Hence $W$ is a wedge of simply-connected spheres with 
${\rm log}~{\rm index}(H_\ast(\Omega W;\mathbb{Q}))\in (0,\infty)$, 
so Theorem~\ref{FHTthm1} implies that $\mathcal{L} W$ has good exponential growth. 
\end{proof}

\begin{proof}[Proof of Theorem~\ref{main}~(a)] 
By hypothesis, $\Omega h$ has a right homotopy inverse,  
${\rm log}~{\rm index}(\pi_{\ast}(Z))<{\rm log}~{\rm index}(\pi_{\ast}(Y))$, 
and ${\rm log}~{\rm index}(H_\ast(\Omega Y;\mathbb{Q}))\in (0,\infty)$. 
The first hypothesis is the condition defining~$f$ as inert and it combines with the second as the 
conditions defining $f$ as strongly inert. Thus Theorem~\ref{case1gexpthm} applies to show 
that $\mathcal{L} Y$ has good exponential growth. 
\end{proof} 

In Theorem~\ref{case1gexpthm} the condition that ${\rm log}~{\rm index}(H_\ast(\Omega Y;\mathbb{Q}))\in (0,\infty)$ holds for a wide range of spaces. Let $Y$ be a simply-connected $CW$-complex satisfying the following three conditions
\begin{itemize}
\item ${\rm log}~{\rm index}(H_{\ast}(Y;\mathbb{Q}))<\infty$,
\item ${\rm log}~{\rm index}(\pi_\ast(Y))>-\infty$,
\item the rational Lusternik-Schnirelmann category of $Y$ is finite.
\end{itemize}
Then by \cite[Theorem 4]{FHT3}, ${\rm log}~{\rm index}(H_\ast(\Omega Y;\mathbb{Q}))\in (0,\infty)$. In particular, the condition is satisfied for any simply-connected rationally hyperbolic finite $CW$-complex.

Theorem~\ref{case1gexpthm} should be compared to \cite[Theorem 4]{FHT4}. 
When $\Sigma A$ is a sphere Theorem \ref{case1gexpthm} is analogous to \cite[Theorem~4]{FHT4} 
but replaces the finiteness condition on a quotient of the homotopy Lie algebra $\pi_{\ast}(Y)\otimes\mathbb{Q}$ 
with a finiteness condition on the log index of $H_{\ast}(\Omega Y;\mathbb{Q})$.  
As an improvement, it generalizes~\cite[Theorem 4]{FHT4} from $\Sigma A$ being a sphere to 
it being any suspension.

\begin{example}\label{exdefold}
As in \cite[Section 4]{FHT4}, let $Z$ be a closed simply-connected manifold of dimension $k+2$ whose rational cohomology algebra is not generated by a single class. If $Y=Z-{pt}$ then there is a homotopy cofibration $S^{k+1}\stackrel{f}{\rightarrow} Y\stackrel{h}{\rightarrow} Z$. By \cite{HL} and \cite[Example 1]{FHT4}, $f$ is strongly inert. Therefore, if ${\rm log}~{\rm index}(H_\ast(\Omega Y;\mathbb{Q}))\in (0,\infty)$ 
then $\mathcal{L}Y$ has good exponential growth by Theorem \ref{case1gexpthm}. Again, compared 
to \cite[Example 1]{FHT4}, the finiteness condition on a quotient of the homotopy Lie algebra 
$\pi_\ast(\Omega Y)\otimes\mathbb{Q}$ is replaced by one on the log index of 
$H_{\ast}(\Omega Y;\mathbb{Q})$.
\end{example}

\section{Exponential growth in $\hlgy{\calL Y;\mathbb{Q}}$: Case II} 
\label{sec:caseII} 
We consider the rational counterparts of spaces considered in \cite[Section 2]{BT1}.
Let $m$ and $n$ be integers such that $1<m\leq n-m$. Suppose that $Y$ is rationally a finite $n$-dimensional $(m-1)$-connected $CW$-complex with rational homology
\[
H_\ast(Y;\mathbb{Q})\cong \mathbb{Q}\{a_1, \ldots, a_{\ell},z\},
\]
where
\[
1<m=|a_1|\leq \cdots |a_\ell|=n-m<|z|=n.
\]
Let $Y_{n-1}$ be the $(n-1)$-skeleton of $Y$ and let $i:Y_{n-1}\rightarrow Y$ be the inclusion. Define $\mathcal{Y}$ as the collection of all such spaces $Y$ which also satisfy the following two properties:
\begin{itemize}
\item[(1)] there is a rational homotopy equivalence $Y_{n-1}\simeq \Sigma J\vee (S^{m}\vee S^{n-m})$ for some space $J$ that is not rationally contractible;
\item[(2)] if $Z$ is the homotopy cofiber of the composite $\Sigma  J\hookrightarrow Y_{n-1}\stackrel{i}{\rightarrow}Y$ then there is a ring isomorphism $H^\ast(Z;\mathbb{Q})\cong H^\ast(S^m\times S^{n-m};\mathbb{Q})$.
\end{itemize} 
It is worth emphasizing that any space $Y\in\mathcal{Y}$ is rationally a simply-connected 
finite $CW$-complex, and therefore the space $\Sigma J$ is also rationally a simply-connected finite 
$CW$-complex.

The following preliminary result is needed; its proof in~\cite{FHT3} using the Sullivan and 
Adams-Hilton models suggests the statement may have been well known beforehand. 

\begin{proposition}[Proposition 2 in~\cite{FHT3}]\label{FHTlogindex} 
Let $X$ be a simply-connected $CW$-complex that is rationally of finite type. Then 
${\rm log}~{\rm index}(\pi_{\ast}(X))<\infty$ if and only if 
${\rm log}~{\rm index}(H_{\ast}(X;\mathbb{Q}))<\infty$.~$\qqed$ 
\end{proposition}  

\begin{corollary}\label{FHTlogindexcor} 
Let $X$ be a simply-connected finite $CW$-complex. Then 
${\rm log}~{\rm index}(H_{\ast}(\Omega X;\mathbb{Q}))<\infty$. 
\end{corollary}  

\begin{proof} 
Since $X$ is a finite $CW$-complex, ${\rm log}~{\rm index}(H_{\ast}(X;\mathbb{Q}))<\infty$, 
so Proposition~\ref{FHTlogindex} implies that 
${\rm log}~{\rm index}(\pi_{\ast}(X))<\infty$. By~(\ref{logcompare}), this implies that 
${\rm log}~{\rm index}(H_{\ast}(\Omega X;\mathbb{Q}))<\infty$. 
\end{proof}

\begin{theorem}\label{ysetgexpthm}
For any $Y\in\mathcal{Y}$, $\mathcal{L}Y$ has good exponential growth.
\end{theorem}
\begin{proof} 
Consider the homotopy cofibration 
\(\nameddright{\Sigma J}{f}{Y}{h}{Z}\) 
that defines $Z$. Note that as $\Sigma J$ and~$Y$ are rationally simply-connected finite 
$CW$-complexes, so is $Z$. By~\cite[Section 2]{BT1}, $\Omega h$ has a rational right homotopy inverse and 
$\Omega Z\simeq\Omega (S^{m}\times S^{n-m})$ rationally. Arguing as in~\cite{BT1}, or alternatively 
by Theorem~\ref{BTtheorem}, there is a rational homotopy fibration 
\[
(\Omega (S^m\times S^{n-m})\wedge\Sigma J)\vee \Sigma J \stackrel{f}{\rightarrow} Y\stackrel{h}{\rightarrow}Z
\]
and a rational homotopy equivalence 
\begin{equation}\label{BTdeceq2}
\Omega Y\simeq \Omega (S^m\times S^{n-m})\times \Omega ((\Omega (S^m\times S^{n-m})\wedge\Sigma J)\vee \Sigma J).
\end{equation} 
Since $\Omega Z\simeq\Omega(S^{m}\times S^{n-1})$, the space $Z$ is rationally elliptic and 
therefore ${\rm log}~{\rm index}(\pi_\ast(Z))=-\infty$. On the other hand, 
$(\Omega (S^m\times S^{n-m})\wedge\Sigma J)\vee \Sigma J$ is rationally a wedge of at least 
two spheres since~$J$ is not rationally contractible, so 
${\rm log}~{\rm index}(\pi_{\ast}((\Omega (S^m\times S^{n-m})\wedge\Sigma J)\vee \Sigma J))>0$. 
The homotopy equivalence in~(\ref{BTdeceq2}) then implies that 
${\rm log}~{\rm index}(\pi_\ast(Y))>0$. Thus the map~$f$ is strongly inert. Further, 
by~(\ref{logcompare}), ${\rm log}~{\rm index}(\pi_\ast(Y))>0$ implies that 
${\rm log}~{\rm index}(H_{\ast}(\Omega Y;\mathbb{Q}))>0$. On the other hand, since~$Y$ 
is a simply-connected finite $CW$-complex, by Corollary~\ref{FHTlogindexcor}, 
${\rm log}~{\rm index}(H_{\ast}(\Omega Y;\mathbb{Q}))<\infty$. Therefore all the hypotheses of 
Theorem~\ref{case1gexpthm} hold, implying that $\calL Y$ has good exponential growth. 
\end{proof} 

A finite $CW$-complex $M$ is a \emph{Poincar\'{e} Duality complex} if $\cohlgy{M}$ satisfies 
Poincar\'{e} Duality. A closed orientable manifold is an example of a Poincar\'{e} Duality 
complex, but many non-manifold examples also exist. 

\begin{example} 
\label{almostwedge} 
Let $M$ be an orientable closed Poincar\'{e} Duality complex of dimension $n$ and connectivity $m-1$ 
whose $(n-1)$-skeleton $M_{n-1}$ has the property that it is rationally homotopy equivalent to a 
wedge of spheres, say 
\[M_{n-1}\simeq_{\mathbb{Q}}\bigvee_{k=1}^{\ell} S^{n_{k}}.\] 
If $k\geq 3$ then there is a rational homotopy equivalence $M_{n-1}\simeq \Sigma J\vee S^{m}\vee S^{n-m}$, 
where $J$ is rationally nontrivial. 
Poincar\'{e} Duality implies that the spheres $S^{m}$ and $S^{n-m}$ may be chosen to have 
a nontrivial cup product (implying that $M\in\mathcal{Y}$) provided that not all cup products are squares. 
So we exclude the cases when $n\in\{4,8,16\}$ and all of $n_{1},\ldots,n_{k}$ equal $\frac{n}{2}$. 
Then $\calL M$ has good exponential growth by Theorem \ref{ysetgexpthm}. 

 As a concrete example, suppose that $M$ is an $(n-1)$-connected closed smooth $2n$-dimensional manifold. 
These manifolds were deeply studied by Wall \cite{Wal} in geometric topology. As  
$M_{2n-1}\simeq\mathop{\bigvee}\limits_{i=1}^{k} S^{n}$, when $k\geq 3$ it follows that $\calL M$ has good exponential growth provided that not all cup products are squares in $H^\ast(M;\mathbb{Q})$.

On the other hand, if $k=2$, Theorem \ref{ysetgexpthm} does not hold. For example, let $M=S^{m}\times S^{n-m}$. 
In this case, $J$ is contractible and $\mathcal{L}(S^{m}\times S^{n-m})$ is rationally elliptic because 
the standard evaluation fibration 
\(\nameddright{\Omega(S^{m}\times S^{n-m})}{}{\mathcal{L}(S^{m}\times S^{n-1})}{}{S^{m}\times S^{n-m}}\) 
induces split exact sequences in homotopy groups and spheres are elliptic. 
Hence, the assumption in Theorem~\ref{ysetgexpthm} that $J$ is not rationally contractible is necessary.  
\end{example}

\section{Exponential growth in $\hlgy{\calL Y;\mathbb{Q}}$: Case III}
\label{sec:caseIII}
In \cite{FHT5}, F\'{e}lix, Halperin and Thomas studied exponential growth with the analytic information of Hilbert series. For a graded vector space $V=\{V_i\}_{i\geq 0}$, its formal {\it Hilbert series} is defined by
\[
V(z)=\sum\limits_{i\geq 0}{\rm dim}~V_i z^i.  
\]
The radius of convergence $\rho_{V}$ of $V(z)$ is defined by $\rho_V=e^{-{\rm log}~{\rm index}(V)}$. Accordingly if $V$ has exponential growth then $0<\rho_V<1$.
If $X$ is a topological space let $X(z)$ be the Hilbert series of $H_\ast(X;\mathbb{Q})$ and let $\rho_X$ be its radius of convergence. Let 
\[\Omega X(\rho_{\Omega X})=\lim_{z\rightarrow\rho_{\Omega X}}\Omega X(z).\]  
As noted in~\cite[Page 2522]{FHT5}, a key condition in their work is the assumption that 
\mbox{$\Omega X(\rho_{\Omega X})=\infty$}; this is satisfied by all known examples of rationally 
hyperbolic, finite, simply-connected $CW$-complexes.  
The following theorem is a stronger form of part (b) of Theorem \ref{main}. 

\begin{theorem}\label{rhoinfthm}
Let $\Sigma A\stackrel{f}{\rightarrow} Y\stackrel{h}{\rightarrow} Z$ be a homotopy cofibration of simply-connected spaces that are rationally of finite type. Suppose that $A$ is not rationally contractible, ${\rm log}~{\rm index}(H_\ast(\Omega Y;\mathbb{Q}))<\infty$, $f$ is inert and $\Omega Z(\rho_{\Omega Z})=\infty$. 
Then $f$ is strongly inert and $\mathcal{L}Y$ has good exponential growth.
\end{theorem}
\begin{proof} 
Since $f$ is inert, $\Omega h$ has a right rational homotopy inverse, so by Theorem~\ref{BTtheorem} there is a rational homotopy fibration
 \[\nameddright{(\Omega Z\wedge\Sigma A)\vee\Sigma A}{} 
           {Y}{h}{Z}\] 
that splits after looping to give a homotopy equivalence 
\[\Omega Y\simeq\Omega Z\times\Omega((\Omega Z\wedge\Sigma A)\vee\Sigma A).\] 
By hypothesis, $\Sigma A$ is not rationally contractible. As $\Omega Z(\rho_{\Omega Z})=\infty$, 
$Z$ is also not rationally contractible. Therefore Lemma~\ref{BTgrowth} implies that 
${\rm log}~{\rm index}(H_\ast(\Omega(\Omega Z\wedge\Sigma A)\vee\Sigma A);\mathbb{Q}))>0$. 
As $\Omega((\Omega Z\wedge\Sigma A)\vee\Sigma A)$ retracts off $\Omega Y$, this implies that 
${\rm log}~{\rm index}(H_\ast(\Omega Y;\mathbb{Q}))>0$. By hypothesis, this log index is also~$<\infty$, 
so we obtain ${\rm log}~{\rm index}(H_\ast(\Omega Y;\mathbb{Q}))\in (0,\infty)$. Thus, if $f$ is strongly 
inert then the hypotheses of Theorem~\ref{case1gexpthm} are satisfied, implying that $\mathcal{L} Y$ 
has good exponential growth. 

It remains to show that $f$ is strongly inert. 
Rationally $\Sigma  A$ is a wedge of spheres, and since~$A$ is not rationally contractible, this wedge 
has at least one sphere as a summand. Suppose that  $\Sigma A\simeq\mathop{\bigvee}\limits_{\alpha} S^{n_{\alpha}}$. Then
\[
\Omega Z\wedge\Sigma A\simeq \mathop{\bigvee}\limits_{\alpha} \big(S^{n_\alpha} \wedge \Omega Z\big).
\] 
For any $\alpha$ we have $H(\Omega(S^{n_\alpha} \wedge \Omega Z);\mathbb{Q})\cong TV$, the tensor algebra on the graded vector space $V$ such that $V_{i}\cong H_{i-n_\alpha+1}(\Omega Z;\mathbb{Q})$. Thus 
\[
\Omega (S^{n_\alpha} \wedge \Omega Z)(z)=\frac{1}{1-z^{n_\alpha-1}\Omega Z(z)}.
\]
By assumption, $\Omega Z(\rho_{\Omega Z})=\infty$, so it follows that for sufficiently small $\varepsilon>0$, the absolute value of $\omega(z) =z^{n_\alpha-1}\Omega Z(z)$ is greater than $1$ for any $|z|\geq \rho_{\Omega Z}-\varepsilon$. However, as the function $f(\omega)=\frac{1}{1-\omega}$ has radius of convergence equal to $1$, we see that $\Omega (S^{n_\alpha} \wedge \Omega Z)(z)=\frac{1}{1-\omega(z)}$ diverges for any $|z|\geq \rho_{\Omega Z}-\varepsilon$. Thus
$\rho_{\Omega (S^{n_\alpha} \wedge \Omega Z)}<\rho_{\Omega Z}$. Hence
\begin{equation} 
  \label{radcompare} 
\rho_{\Omega((\Omega Z\wedge\Sigma A)\vee\Sigma A)}\leq \rho_{\Omega(\Omega Z\wedge\Sigma A)}
\leq
\rho_{\Omega (S^{n_\alpha} \wedge \Omega Z)}<\rho_{\Omega Z}.
\end{equation}  
By definition $\rho_{\Omega X}=e^{-{\rm log}~{\rm index}(H_{\ast}(\Omega X;\mathbb{Q}))}$,  
so~(\ref{radcompare}) implies that 
\[{\rm log}~{\rm index}(H_{\ast}(\Omega Z;\mathbb{Q}))< 
     {\rm log}~{\rm index}(H_{\ast}(\Omega((\Omega Z\wedge\Sigma A)\vee\Sigma A)).\] 
Since $\Omega((\Sigma\Omega Z\wedge\Sigma A)\vee\Sigma A)$ is a retract of $\Omega Y$, 
we obtain 
\[{\rm log}~{\rm index}(H_{\ast}(\Omega Z;\mathbb{Q})< {\rm log}~{\rm index}(H_{\ast}(\Omega Y;\mathbb{Q}).\] 
By~(\ref{logcompare}), this implies that 
${\rm log}~{\rm index}(\pi_{\ast}(Z))<{\rm log}~{\rm index}(\pi_{\ast}(Y))$, 
and therefore $f$ is strongly inert. 
\end{proof}

\begin{proof}[Proof of Theorem~\ref{main}~(b)] 
We aim to apply Theorem~\ref{rhoinfthm}. By hypothesis, $A$ is not rationally contractible. 
By hypothesis, $Y$ is a simply-connected finite $CW$-complex so Corollary~\ref{FHTlogindexcor}  
implies that ${\rm log}~{\rm index}(H_\ast(\Omega Y;\mathbb{Q}))<\infty$. The hypothesis that $\Omega h$ 
has a right homotopy inverse is, by definition, the same as $f$ being inert. By hypothesis, 
$Z$ is a simply-connected rationally hyperbolic finite $CW$-complex with a finitely generated 
homotopy Lie algebra, so by \cite[Proposition~2.1]{FHT5} or~\cite{A} it follows that 
$\Omega Z(\rho_{\Omega Z})=\infty$. Hence the hypotheses of Theorem~\ref{rhoinfthm} 
hold, implying that $\mathcal{L}Y$ has good exponential growth.
\end{proof}

Theorem \ref{rhoinfthm} is a generalization of \cite[Theorem 1.3]{FHT5}. As a slight improvement, 
\cite{FHT5} uses the condition that $\rho(L_{Z})<\rho(\frac{L_{Z}}{[L_{Z},L_{Z}]})$, where $L_{Z}$ 
is the homotopy Lie algebra of $Z$, in order to show that $\Omega Z(\rho_{\Omega Z})=\infty$ whereas 
we start directly from the putatively weaker condition. As a larger improvement, \cite[Theorem 1.3]{FHT5} 
holds only for the case when $\Sigma A$ is a sphere, whereas in our case any suspension will do.

To give examples of Theorem~\ref{rhoinfthm} an additional result of F\'{e}lix, Halperin and Thomas 
is needed, which was established within the proof of \cite[Theorem 1.4]{FHT5}. 

\begin{lemma}\label{wedgeexplemma}
Let $M$ and $N$ be simply-connected $CW$-complexes of finite type that are not rationally contractible. If $\rho_{\Omega N}\leq\rho_{\Omega M}$ and $\Omega N(\rho_{\Omega N})=\infty$ then $\rho_{\Omega (M\vee N)}<\rho_{\Omega N}$ and $\Omega (M\vee N)(\rho_{\Omega (M\vee N)})=\infty$.~$\qqed$ 
\end{lemma} 

Notice, incidentally, that Lemma \ref{wedgeexplemma} can be strengthened by Theorem \ref{rhoinfthm} 
if $M$ is a suspension and $H_{\ast}(\Omega(M\vee N);\mathbb{Q})$ has exponential growth. 

\begin{lemma}
Let $\Sigma M$ and $N$ be simply-connected $CW$-complexes of finite type that are not rationally 
contractible. If ${\rm log}~{\rm index}(H_{\ast}(\Omega(\Sigma M\vee N);\mathbb{Q})<\infty$ and 
$\Omega N(\rho_{\Omega N})=\infty$, then $\mathcal{L}(\Sigma M\vee N)$ has good exponential 
growth and $\rho_{\Omega (\Sigma M\vee N)}<\rho_{\Omega N}$.  
\end{lemma}
\begin{proof}
Consider the homotopy cofibration $\Sigma M \stackrel{i_1}{\rightarrow} \Sigma M\vee N\stackrel{q_2}{\rightarrow} N$, where $i_1$ is the inclusion of the first factor and $q_2$ is the projection onto the second factor. Then $i_1$ is inert because $\Omega q_{2}$ has a right homotopy inverse.  By hypothesis, 
${\rm log}~{\rm index}(H_{\ast}(\Omega(\Sigma M\vee N);\mathbb{Q})<\infty$ 
and $\Omega N(\rho_{\Omega N})=\infty$, so Theorem \ref{rhoinfthm} implies that $i_{1}$ is strongly inert and $\mathcal{L}(\Sigma M\vee N)$ has good exponential growth. 

Since $i_{1}$ is strongly inert, 
${\rm log}~{\rm index}(\pi_{\ast}(N))<{\rm log}~{\rm index}(\pi_{\ast}(\Sigma M\vee N))$. 
By~(\ref{logcompare}), this is equivalent to 
${\rm log}~{\rm index}(H_{\ast}(\Omega N;\mathbb{Q}))< 
     \mbox{${\rm log}~{\rm index}(H_{\ast}(\Omega(\Sigma M\vee N);\mathbb{Q}))$}$. 
By definition, $\rho_{\Omega X}=e^{-{\rm log}~{\rm index}(H_{\ast}(\Omega X;\mathbb{Q}))}$, 
so we obtain $\rho_{\Omega (\Sigma M\vee N)}<\rho_{\Omega N}$. 
\end{proof}

Suppose that there are homotopy cofibrations $\Sigma A\stackrel{f}{\rightarrow} X\stackrel{j}{\rightarrow} M$ and $\Sigma A\stackrel{g}{\rightarrow} Y\stackrel{k}{\rightarrow} N$. {\it The generalized connected sum $M\mathop{\conn}\limits_{\Sigma A} N$} over $\Sigma A$, introduced in \cite[Section 8]{T}, is defined by the homotopy cofibration
\[
\Sigma A\stackrel{f+g}{\longrightarrow} X\vee Y\longrightarrow M\mathop{\conn}\limits_{\Sigma A} N.
\]

\begin{theorem}\label{sumgoodthm}
Let $M\mathop{\conn}\limits_{\Sigma A} N$ be the generalized connected sum of simply-connected $CW$-complexes of finite type such that ${\rm log}~{\rm index}(H_\ast(\Omega(M\mathop{\conn}\limits_{\Sigma A} N);\mathbb{Q}))<\infty$. Suppose that $f$ and $g$ are inert, and $A$ is not rationally contractible.
If $\rho_{\Omega N}\leq\rho_{\Omega M}$ and $\Omega N(\rho_{\Omega N})=\infty$ then $\mathcal{L}(M\mathop{\conn}\limits_{\Sigma A} N)$ has good exponential growth.
\end{theorem}
\begin{proof} 
By definition, $f+g$ is the composite 
\(\nameddright{\Sigma A}{\sigma}{\Sigma A\vee\Sigma A}{f\vee g}{X\vee Y}\) 
where $\sigma$ is the comultiplication. From the composition we obtain a homotopy cofibration diagram 
\[\diagram 
      \Sigma A\rto^-{\sigma}\ddouble & \Sigma A\vee\Sigma A\rto\dto^{f\vee g} & \Sigma A\dto^{\mathfrak{f}} \\ 
      \Sigma A\rto^-{f+g} & X\vee Y\rto\dto^{j\vee k} & M\mathop\conn\limits_{\Sigma A} N\dto^{q} \\ 
      & M\vee N\rdouble & M\vee N 
  \enddiagram\] 
that defines the maps $\mathfrak{f}$ and $q$. Intuitively, $\mathfrak{f}$ maps to the ``collar" 
in the connected sum and $q$ collapses it. Since $f$ and $g$ are both inert, $\Omega j$ and 
$\Omega k$ have right rational homotopy inverses. Arguing as in~\cite[Proposition 6.1]{BT2} 
then implies that $\Omega q$ has a right homotopy inverse. (The argument in~\cite{BT2} had 
$\Sigma A$ a sphere, but it works equally well in this more general setting.) Thus 
$\mathfrak{f}$ is inert. Since $\rho_{\Omega N}\leq\rho_{\Omega M}$ and 
$\Omega N(\rho_{\Omega N})=\infty$, Lemma~\ref{wedgeexplemma} implies that 
$\Omega (M\vee N)(\rho_{\Omega (M\vee N)})=\infty$. By hypothesis, $A$ is not rationally 
contractible and 
${\rm log}~{\rm index}(H_\ast(\Omega(M\mathop{\conn}\limits_{\Sigma A} N);\mathbb{Q}))<\infty$. 
Therefore all the hypotheses of Theorem~\ref{rhoinfthm} are satisfied when 
applied to the homotopy cofibration  
\(\nameddright{\Sigma A}{\mathfrak{f}}{M\mathop\conn\limits_{\Sigma A} N}{q}{M\vee N}\), 
implying that $\calL(M\mathop\conn\limits_{\Sigma A} N)$ has good exponential growth. 
\end{proof} 

Theorem~\ref{sumgoodthm} partially generalizes a result of Lambrechts~\cite[Theorem 3]{L2} and its improvement by F\'{e}lix-Halperin-Thomas \cite[Theorem 1.4]{FHT5}. 
They both considered the connected sum 
$M\conn N$ of two $n$-dimensional closed simply-connected manifolds. Lambrechts showed 
that if the rational cohomology of $M$ or $N$ is not generated by a single class then $\calL(M\conn N)$ 
has exponential growth. F\'{e}lix-Thomas-Halperin improved on this by showing that if 
the rational cohomology of $M$ is not generated by a single class while~$N$ is not rationally a sphere, 
then $\calL(M\conn N)$ has good exponential growth. By comparison, in the special case of $\Sigma A=S^{n-1}$ 
and $M$ and $N$ being $n$-dimensional closed simply-connected manifolds, the generalized 
connected sum $M\mathop{\conn}\limits_{S^{n-1}} N$ is the usual connected sum $M\conn N$. If 
both the rational cohomology algebras of $M$ and~$N$ are not generated by a single class 
then the attaching maps $f$ and~$g$ are inert by \cite[Theorem 5.1]{HL}. Since $M\conn N$ 
is a simply-connected finite $CW$-complex, Corollary~\ref{FHTlogindexcor} implies that 
${\rm log}~{\rm index}(H_{\ast}(\Omega(M\conn N);\mathbb{Q}))<\infty$. Therefore Theorem~\ref{sumgoodthm} implies $\calL(M\conn N)$ has good exponential growth.

\section{Torsion growth in homotopy groups} 
\label{sec:local}
This section turns from growth in the rational homology of free loop spaces to growth of 
torsion in homotopy groups. This involves homotopy exponents, mod-$p^{r}$ hyperbolicity  
and Moore's Conjecture. 

Of particular relevance in our case is a wedge $S^{m}\vee S^{n}$ of spheres. By the Hilton-Milnor 
Theorem $\Omega(S^{m}\vee S^{n})$ is homotopy equivalent to an infinite product of loops 
on spheres of arbitrarily large dimension. It is known that the exponent of a sphere increases 
with the dimension. Consequently, $\Omega(S^{m}\vee S^{n})$ is rationally hyperbolic and has 
no exponent at any prime $p$, and therefore satisfies Moore's Conjecture. Boyde~\cite{B} 
went further by showing that $\Omega(S^{m}\vee S^{n})$ is mod-$p^{r}$ hyperbolic for all 
primes~$p$ and all $r\geq 1$. 

In Theorem~\ref{mostpwedge} we will give a new class of spaces that satisfies Moore's 
Conjecture for all but finitely many primes and, for those primes, is mod-$p^{r}$ hyperbolic. 
This requires a preliminary lemma that is a $p$-local approximation to the statement that 
any suspension is rationally a wedge of spheres.

\begin{lemma} 
   \label{primestosusp} 
   Let $X$ be a path-connected finite $CW$-complex of dimension $d$ and connectivity $s$. 
   Let~$p$ be a prime such that $p>\frac{1}{2}(d-s+1)$ and $H_{\ast}(X;\mathbb{Z})$ is $p$-torsion 
   free. Then $\Sigma X$ is $p$-locally homotopy equivalent to a wedge of spheres. 
\end{lemma} 

\begin{proof} 
Take homology with integer coefficients. 
Since $\Sigma X$ is simply-connected and of dimension~$d+1$ it has a homology decomposition, 
which as in~\cite[Chapter 4.H]{H} is a sequence of homotopy cofibrations 
\[\nameddright{M_{t}}{f_{t}}{(\Sigma X)_{t-1}}{}{(\Sigma X)_{t}}\] 
for $2\leq t\leq d+1$ with $(\Sigma X)_{1}$ equal to the basepoint, $(\Sigma X)_{d+1}=\Sigma X$, 
$M_{t}$ is a wedge of 
$t-1$ dimensional spheres and $t$ dimensional Moore spaces, and the attaching map $f_{t}$ 
has the property that it induces the zero map in homology. Notice that as $(f_{t})_{\ast}=0$ 
there is an isomorphism $H_{\ast}((\Sigma X)_{t})\cong H_{\ast}((\Sigma X)_{t-1})\oplus H_{\ast}(\Sigma M_{t})$, 
which iteratively implies that $H_{\ast}((\Sigma X)_{t})$ is a direct summand of $H_{\ast}(\Sigma X)$. 
The assumption that $H_{\ast}(X;\mathbb{Z})$ is $p$-torsion free therefore implies that 
$H_{\ast}((\Sigma X)_{t})$ is $p$-torsion free, and therefore 
$H_{\ast}(\Sigma M_{t})$ is also $p$-torsion free. Hence $H_{\ast}(M_{t})$ is $p$-torsion free. Thus 
the $p$-localization of each $M_{t}$ is a wedge of spheres, say $M_{t}\simeq\bigvee_{i=1}^{k_{t}} S^{t-1}$.  
Therefore, $p$-locally, the homotopy cofibrations in the homology decomposition of $\Sigma X$ are of the form 
\[\nameddright{\bigvee_{i=1}^{k_{t}} S^{t-1}}{f_{t}}{(\Sigma X)_{t-1}}{}{(\Sigma X)_{t}},\] 
which is the usual skeletal filtration of $\Sigma X$, but with the extra property that $(f_{t})_{\ast}=0$ 
for each~$t$.

Localize at $p$. As $X$ is $s$-connected, $\Sigma X$ is $(s+1)$-connected, so each of 
$(\Sigma X)_{1},\ldots,(\Sigma X)_{s+1}$ is contractible and $(\Sigma X)_{s+2}$ is homotopy 
equivalent to a wedge of spheres. Suppose inductively that there is a $p$-local homotopy equivalence 
$(\Sigma X)_{t-1}\simeq\mathop{\bigvee}\limits_{\alpha\in\mathcal{I}} S^{k_{\alpha}}$ for $s+2\leq k_{\alpha}\leq t-1$. 
Rationally, any suspension is homotopy equivalent to a wedge of spheres. Thus $f_{t}$ is 
rationally trivial. This implies that the obstructions to~$f_{t}$ being $p$-locally null homotopic are: 
(i) instances where $f_{t}$ has degree~$p^{r}$ for some $r\geq 0$ on an $S^{t-1}$ summand, 
and (ii) torsion homotopy classes in 
$\pi_{t-1}((\Sigma X)_{t-1})=\pi_{t-1}(\mathop{\bigvee}\limits_{\alpha\in\mathcal{I}} S^{k_{\alpha}})$. 
Since $(f_{t})_{\ast}=0$, (i) cannot occur. For (ii), 
the least nontrivial $p$-torsion class in $\pi_{\ast}(S^{n})$ occurs in dimension $n+2p-3$. 
As this number increases with $n$, the Hilton-Milnor Theorem implies that the least 
nontrivial $p$-torsion class in $\pi_{\ast}(\mathop{\bigvee}\limits_{\alpha\in\mathcal{I}} S^{k_{\alpha}})$ 
occurs in dimension $s+2p-1$. Thus if $t-1<s+2p-1$ then there is no $p$-torsion in 
$\pi_{t-1}(\mathop{\bigvee}\limits_{\alpha\in\mathcal{I}} S^{k_{\alpha}})$. Consequently, $f_{t}$ 
is null homotopic, implying that $(\Sigma X)_{t}$ is $p$-locally homotopy equivalent to a wedge of spheres. 
By induction, $\Sigma X$ will be $p$-locally homotopy equivalent to a wedge of spheres provided 
that $d<s+2p-1$, or equivalently, provided that $p>\frac{1}{2}(d-s+1)$. 
\end{proof} 

If $X$ is a path-connected finite $CW$-complex of dimension $d$ and connectivity $s$, 
let $\calP(X)$ be the set of primes $q$ such that $q\leq\frac{1}{2}(d-s+1)$ or 
$H_{\ast}(X;\mathbb{Z})$ has $q$-torsion. Note that the finiteness condition on $X$ implies that 
$\mathcal{P}(X)$ is a finite set of primes. Lemma~\ref{primestosusp} implies that if we 
localize away from $\calP(X)$ then~$\Sigma X$ is homotopy equivalent to a wedge of spheres.  

\begin{theorem} 
   \label{mostpwedge} 
   Let 
   \(\nameddright{\Sigma A}{f}{Y}{h}{Z}\) 
   be a homotopy cofibration of simply-connected finite $CW$-complexes such that 
   $A$ and $Z$ are not rationally contractible. If $\Omega h$ has a right homotopy inverse,   
   then localized away from $\calP=\calP(A)\cup\calP(Z)$ there is a 
   retraction of $\Omega(S^{m}\vee S^{n})$ off $\Omega Y$ for some $m,n\geq 2$. 
\end{theorem} 

\begin{proof} 
Since $\Omega h$ has a right homotopy inverse, by Theorem~\ref{BTtheorem} 
there is a homotopy fibration 
\[\nameddright{(\Omega Z\wedge\Sigma A)\vee\Sigma A}{}{Y}{h}{Z}\] 
and a homotopy equivalence 
\begin{equation} 
  \label{LoopXdecomp} 
  \Omega Y\simeq\Omega Z\times\Omega((\Omega Z\wedge\Sigma A)\vee\Sigma A). 
\end{equation} 
Localize away from $\calP$. By Lemma~\ref{primestosusp}, $\Sigma A$ is homotopy equivalent 
to a wedge of spheres. Since~$\Sigma A$ is not rationally contractible it has at least one sphere 
as a wedge summand. Let $S^{m}$ be a sphere of least dimension in this wedge decomposition. 
Notice that $m\geq 2$ since $\Sigma A$ is simply-connected. At this point we have $S^{m}$ 
retracting off $\Sigma A$ and $\Omega Z\wedge S^{m}$ retracting off $\Omega Z\wedge\Sigma A$. 

Now consider the map  
\[\llnamedright{\Omega Z\wedge S^{m}\simeq\Sigma^{m}\Omega Z}{\Sigma^{m-1} ev}{\Sigma^{m-1} Z}\] 
where $ev$ is the canonical evaluation. As $m\geq 2$, this map makes sense and $\Sigma^{m-1} Z$ 
is a suspension. As we are localized away from $\calP$, by Lemma~\ref{primestosusp} the 
space $\Sigma^{m-1} Z$ is homotopy equivalent to a wedge of spheres. Since $Z$ is not 
rationally contractible it has at least one sphere as a wedge summand. Let $S^{n}$ be a sphere 
of least dimension in this wedge decomposition.  
We claim that this sphere also retracts off $\Omega Z\wedge S^{m}$. If so then 
$S^{n}\vee S^{m}$ retracts off $(\Omega Z\wedge\Sigma A)\vee\Sigma A$ and 
hence~(\ref{LoopXdecomp}) implies that $\Omega(S^{m}\vee S^{n})$ retracts off $\Omega Y$. 

It remains to show that $S^{n}$ retracts off $\Omega Z\wedge S^{m}$. Take homology 
with $\mathbb{Z}_{(\frac{1}{\calP})}$-coefficients. Note that $n$ is the least 
degree for which $\Sigma^{m-1} Z$ has nontrivial homology. Since $Z$ is simply-connected, 
the Serre exact sequence applied to the homotopy fibration 
\(\nameddright{\Sigma\Omega Z\wedge\Omega Z}{}{\Sigma\Omega Z}{ev}{Z}\) 
implies that $ev$ induces an isomorphism in the least nontrivial degree in homology. Therefore 
$\Sigma^{m-1} ev$ induces an isomorphism in degree $n$. On the other hand, the Hurewicz 
Theorem implies that, for some finite number $\ell\geq 1$, there is a map 
\(g\colon \namedright{\bigvee_{i=1}^{\ell} S^n}{}{\Omega Z\wedge S^{m}}\)   
that induces an isomorphism in degree~$n$ homology. Thus $(\Sigma^{m-1} ev)\circ g$ 
induces an isomorphism in degree $n$ homology. Since $\Sigma^{m-1} Z$ is homotopy 
equivalent to a wedge of spheres, this implies that $(\Sigma^{m-1} ev)\circ g$ has a left 
homotopy inverse. Hence there is a retraction of $S^{n}$ off $\Sigma Z\wedge S^{m}$. 
\end{proof} 

\begin{corollary} 
   \label{mostpMoore} 
   With hypotheses as in Theorem~\ref{mostpwedge}, the space $Y$ has the following properties: 
   \begin{letterlist} 
      \item $Y$ is rationally hyperbolic; 
      \item $Y$ has no homotopy exponent at any prime $p\notin\calP$; 
      \item $Y$ is mod-$p^{r}$ hyperbolic for all primes $p\notin\calP$ and all $r\geq 1$.  
   \end{letterlist} 
   Consequently, for all but finitely many primes, $Y$ satisfies Moore's Conjecture and 
   is mod-$p^{r}$ hyperbolic. 
\end{corollary} 

\begin{proof} 
Parts~(a) to~(c) follow because they are satisfied by $S^{m}\vee S^{n}$ and 
because $\Omega(S^{m}\vee S^{n})$ retracts off $\Omega Y$. 
\end{proof}

\begin{proof}[Proof of Theorem~\ref{main}~(c)] 
Suppose that $Y$ has dimension $d$ and connectivity $s$. By hypothesis, 
$H_{\ast}(Y;\mathbb{Z})$ is $p$-torsion free, so if $p>\frac{1}{2}(d-s+1)$ then 
$p\notin\mathcal{P}(Y)$. Therefore Corollary~\ref{mostpMoore} implies that 
$Y$ is rationally hyperbolic and $\mathbb{Z}/p^{r}$-hyperbolic for all $r\geq 1$. 
\end{proof}

\begin{example} 
We revisit the class $\mathcal{Y}$ in Section~\ref{sec:caseII}.
Let $Y\in\mathcal{Y}$ and consider the associated homotopy cofibration 
\(\nameddright{\Sigma J}{f}{Y}{h}{Z}\). 
By the hypotheses on $\mathcal{Y}$, the space $J$ is not rationally contractible, and~$Z$ 
has the rational homology of $S^{m}\times S^{n-m}$. By~\cite{BT1}, $\Omega h$ has a 
right rational homotopy inverse. By Lemma~\ref{primestosusp}, 
$\Sigma J$ is homotopy equivalent to a wedge of spheres after localization away from 
$\calP(J)$. Now Theorem~\ref{mostpwedge} applies and so Corollary~\ref{mostpMoore} 
implies that $Y$ is hyperbolic, has no exponent at any prime $p\notin\calP(J)\cup\calP(Z)$, 
and for any such prime~$Y$ is mod-$p^{r}$ hyperbolic for all $r\geq 1$. 
\end{example}

Finally, observe that as well as Moore's Conjecture, Theorem~\ref{mostpwedge} is closely 
linked to the Vigu\'{e}-Poirrier Conjecture. The hypothesis that $\Omega h$ has a right 
homotopy inverse implies that $f$ is inert. 
Since $Y$ is a simply-connected finite $CW$-complex, by Corollary~\ref{FHTlogindexcor},
${\rm log}~{\rm index}(H_\ast(\Omega Y;\mathbb{Q}))<\infty$.
Therefore, if either $f$ is strongly inert or $\Omega Z(\rho_{\Omega Z})=\infty$, then Theorems~\ref{case1gexpthm} 
or~\ref{rhoinfthm} respectively implies that $\calL Y$ has good exponential growth.

\bibliographystyle{amsalpha}

\end{document}